\documentclass[11pt,twoside,english]{article}
\usepackage{amsmath}
\usepackage{amsfonts}
\usepackage{mathrsfs}
\usepackage{color}
\usepackage{ amsmath, amsfonts, amssymb, amsthm, amscd}
\usepackage[T1]{fontenc}
\usepackage[english]{babel}
\usepackage[noinfoline]{imsart}

\newtheorem{theorem}{Theorem}
\newtheorem{lemma}{Lemma}
\newtheorem{proposition}{Proposition}

\newtheorem{definition}{Definition}

\newtheorem{remark}{Remark}

\newcommand{\cA}{\ensuremath{\mathcal A}}

\newcommand{\cD}{\ensuremath{\mathcal D}}
\newcommand{\cE}{\ensuremath{\mathcal E}}
\newcommand{\cF}{\ensuremath{\mathcal F}}

\newcommand{\cH}{\ensuremath{\mathcal H}}

\newcommand{\cO}{\ensuremath{\mathcal O}}

\newcommand{\cU}{\ensuremath{\mathcal U}}

\newcommand{\bbR}{{\ensuremath{\mathbb R}} }

\newcommand{\be}{\begin{equation}}
\newcommand{\ee}{\end{equation}}
\newcommand{\beq}{\begin{eqnarray}}
\newcommand{\eeq}{\end{eqnarray}}

\newcommand{\R}{\mathbb{R}}

\newcommand{\ced}{\end{proof}}

\begin{document}
\begin{frontmatter}
\title{The Obstacle Problem for Quasilinear Stochastic PDEs with Neumann boundary condition}
\date{}
\runtitle{}
\author{\fnms{Yuchao}
 \snm{DONG}\corref{}\ead[label=e1]{ycdong@fudan.edu.cn}}
\address{Fudan University
\\\printead{e1}}
\author{\fnms{Xue}
 \snm{YANG}\corref{}\ead[label=e2]{xyang2013@tju.edu.cn}}
\address{Tianjin University
\\\printead{e2}}
\author{\fnms{Jing}
 \snm{ZHANG}\corref{}\ead[label=e3]{zhang\_jing@fudan.edu.cn}}
\address{Fudan University
\\\printead{e3}}

\runauthor{Y. Dong, X. Yang and J. Zhang}

\begin{abstract}
We prove the existence and uniqueness of solution of the obstacle problem for quasilinear stochastic partial differential equations (OSPDEs for short) 
with Neumann boundary condition. Our method is based on the analytical technics coming from parabolic potential theory. The solution is expressed 
as a pair $(u,\nu)$ where $u$ is a predictable continuous process which takes values in a proper Sobolev space and $\nu$ is a random regular 
measure satisfying minimal Skohorod condition.
\end{abstract}


\end{frontmatter}

%
%
%
%
%
%
%
%
%
%
%
%
%

\section{Introduction}
Our aim is to prove the existence and uniqueness result for the following OSPDE with Neumann boundary condition that we write formally as:
\begin{equation}\label{SPDEO}\left\{ \begin{split}&du_t(x)=\partial_i  \left(a_{i,j}(x)\partial_ju_t(x)+g_i(t,x,u_t(x),\nabla u_t(x))\right)dt+f(t,x,u_t(x),\nabla u_t(x))dt \\&\quad \quad\ \  \ \ \ \ +\sum_{j=1}^{+\infty}h_j(t,x,u_t(x),\nabla u_t(x))dB^j_t, \quad\forall (t,x) \in  \mathbb R^+\times \mathcal O,\\
&\sum_{i,j} (a_{i,j}\partial_ju_t(x)+g_i(t,x,u_t(x),\nabla u_t(x)))n_i(x)=l(t,x,u_t(x)), \quad\forall (t,x) \in  \mathbb R^+\times \partial \mathcal O,\\
&u_t(x)\geq S_t(x),\quad \forall (t,x)\in \mathbb R^+\times \mathcal O, \\& u_0(x)=\xi(x),\quad\forall x\in\mathcal O, \end{split}\right.\end{equation}
where $\vec n(x)=(n_1(x),n_2(x),...,n_d(x))$ is the outer normal to $\cO\subset \bbR^d$ at $x\in \partial D$ and
we also use the notation $\frac{\partial u_t}{\partial v_L}:=\sum_{i,j} a_{i,j}\partial_ju_t(x)n_i(x).$
$a=(a_{i,j})_{1\leq i,j\leq d}$ is a symmetric  bounded measurable matrix which defines a second order operator on
$\mathcal{O}$, with Neumann boundary condition. The initial
condition is given as $u_0=\xi$, a $L^2(\mathcal{O})-$valued random
variable, and $f$, $g=(g_1,...,g_d)$ and $h=(h_1,...h_i,...)$ are
non-linear random functions. Given an obstacle $S: \Omega\times[0,T]\times \cO \rightarrow \R$, we study the OSPDE \eqref{SPDEO},
i.e. we want to find a solution that satisfies "$u\geq S$" where the obstacle $S$ is regular in some sense  and
controlled by the solution of a stochastic partial differential equation (SPDE for short).
\vspace{2mm}

The obstacle problems for SPDEs have been widely studied since Nualart and Pardoux   obtained the existence and uniqueness of solution for the obstacle problem with the diffusion matrix $a=I$ (\cite{NualartPardoux}). Then  Donati-Martin and Pardoux \cite{DonatiPardoux} proved it for the general diffusion matrix.  Various properties of  the solutions were studied later in \cite{DMZ}, \cite{XuZhang}, \cite{Zhang} etc..
As backward stochastic differential equations (BSDEs for short) were rapidly developed, obstacle problem associated with a nonlinear partial differential equation (PDE for short) with more general coefficients, and the properties of the solutions for OSPDEs were studied in the framework of BSDEs (see \cite{EKPPQ}, \cite{BCEF}, \cite{MatoussiStoica}, etc.).

\vspace{2mm}

This article is motivated by  a serial research of the obstacle problems (\cite{DMZ12}, \cite{MatoussiStoica}, etc.). The approach used in \cite{MatoussiStoica} is based on the probabilistic interpretation of the solution presented by backward doubly stochastic differential equations (BDSDEs for short).  In \cite{DMZ12}, the authors studied the OSPDE  with null Dirichlet condition on an open domain in $\mathbb{R}^d$. Their method was based on the technics of parabolic potential theory developed by M. Pierre (\cite{Pierre}, \cite{PIERRE}).  Inspired by this analytic method, we will continue the study on this obstacle problem but with Neumann boundary condition in this paper. The key point is to construct a solution which admits a quasi-continuous version defined outside a polar set and the regular measures which in general are not absolutely continuous w.r.t. the Lebesgue measure, do not charge polar sets. Due to the Neumann boundary condition is considered, the boundary behavior of the solution should be estimated, and thanks to trace theorem, this term can be controlled, with leads to the contraction condition on the coefficients should be strengthened in some sense.
\vspace{2mm}

The remainder of this paper is organized as follows: in the second section, we set the assumptions then we recall in the third section some useful results on SPDEs with Neumann boundary condition. In Section 4, we introduce the notion of regular measures associated to parabolic potentials. The quasi-continuity of the solution for SPDE without obstacle is proved in the fifth section. The last section is devoted to proving the existence and uniqueness result, and It\^o's formula and comparison theorem are also established.

\section{Preliminaries}
We consider a sequence $((B^i(t))_{t\geq0})_{i\in\mathbb{N}^*}$ of
independent Brownian motions defined on a standard filtered
probability space $(\Omega,\mathcal{F},(\mathcal{F}_t)_{t\geq0},P)$.

Let $\mathcal{O}\subset \bbR^d$ be a bounded open domain  with a smooth boundary and
$L^2(\mathcal{O})$ the set of square integrable functions with
respect to the Lebesgue measure on $\mathcal{O}$, which is an Hilbert
space equipped with the  scalar product and norm as follows
$$(u,v)=\int_\mathcal{O}u(x)v(x)dx,\qquad\| u\|=(\int_\mathcal{O}u^2(x)dx)^{1/2}.$$
We assume the boundary $\partial \cO$ of $\cO$ is smooth, and define a Hilbert space $L^2(\partial \cO)$ with the scalar product and norm as follows
$$(u,v)_{L^2(\partial \cO)}=\int_{\partial \cO}u(x)v(x)d\sigma(x),\qquad \|u\|_{L^2(\partial \cO)}=(\int_{\partial \cO}u^2(x)d\sigma(x))^{1/2},$$
where $d\sigma(x)$ is the Lebesgue measure in $\mathbb{R}^{d-1}$.

For the simplicity of the notations, we sometimes also denote $(\cdot,\cdot)_{L^2(\partial \cO)}$ and $\|\cdot\|_{L^2(\partial \cO)}$ by $(\cdot,\cdot)$ and $\|\cdot\|$. But it will be clear which space we are referring  to from the context.

We assume that
$a=(a_{i,j})_{1\leq i,j\leq d}$ is a smooth symmetric matrix defined on
$\bar {\mathcal{O}}$  satisfying the uniform ellipticity
condition$$\lambda|\xi|^2\leq\sum_{i,j=1}^d
a_{i,j}(x)\xi^i\xi^j\leq\Lambda|\xi|^2,\ \forall x\in\bar{\mathcal{O}},\
\xi\in \bbR^d,$$where $\lambda$ and $\Lambda$ are positive constants.

Consider the second order generator $(\cA, \cD(\cA))$ with null Neumann boundary condition, where
$$\cD(\cA)=\{u\in L^2(\cO)|\cA u\in L^2(\cO)\ \mbox{and} \ \sum_{i,j} a_{i,j}\partial_ju\,n_i=0 \text{ on } \partial \cO\}$$
and
$$\cA u=div(a_{i,j}\nabla u)=\sum_{i,j} \frac{\partial}{\partial x_i}(a_{i,j}\frac{\partial u}{\partial x_j})\,.$$
 Let $(F,\mathcal{E})$ be the associated Dirichlet form given by
$$F:=H^1(\mathcal{O}):=\{u\in L^2(\mathcal{O})| \nabla u\in L^2(\mathcal{O})\ \mbox{in weak sense}\}$$ and
$$\mathcal{E}(u,v):=\sum_{i,j=1}^{n}(a_{i,j}\partial_j u,\partial_j v)\ \mbox{and}\ \mathcal{E}(u)=\cE(u,u),\ \forall u,v\in F.$$
The  semigroup associated with $(\cA, \cD(\cA))$ on $L^2(\cO)$ is denoted by $\{P_t\}_{t\geq 0}$.

We have the following theorem (Theorem 7.26 in \cite{GT} and Theorem 5.22 in \cite{Ada}):
\begin{theorem}
Let $\cO$ be a smooth bounded domain in $\mathbb R^d$. Then
\begin{itemize}
\item[(1.)] {if $0 \le m <k-\frac{d}{2}<m+1$, the space $H^k(\cO)$ is continuously imbedded in $C^{m,\alpha}(\bar {\cO})$, with $\alpha=k-d/2-m$;}
\item[(2.)]{The trace operator $Tr$ admits a continuous extension from $H^1(\cO)$ to $L^2(\partial \cO)$.}
\end{itemize}
\end{theorem}
Note that for any $l \in L^2(\partial O)$, one can associate it with an element $\tilde l \in H^{-1}(\cO)$ (the dual space of $H^1(\cO)$) by $\tilde l(u)=\int_{\partial \cO} l(x)u(x)d\sigma(x)$, for any $u \in H^1(\cO)$. Due to the continuity of the trace operator, we have $\|\tilde l\|_{H^{-1}} \le \|Tr\|\|l\|_{L^2(\partial \cO)}$.

We assume that we have predictable random
functions\begin{eqnarray*}&&f:\bbR^+\times\Omega\times\mathcal{O}\times
\bbR\times \bbR^d\rightarrow
\bbR,\\&&g=(g_1,...,g_d):\bbR^+\times\Omega\times\mathcal{O}\times \bbR\times
\bbR^d\rightarrow
\bbR^d,\\&&h=(h_1,...,h_i,...):\bbR^+\times\Omega\times\mathcal{O}\times
\bbR\times \bbR^d\rightarrow \bbR^{\mathbb{N}^*},\\
&&l:\bbR^+\times\Omega\times\partial\mathcal{O}\times
\bbR\rightarrow
\bbR\,.
\end{eqnarray*}In the sequel,
$|\cdot|$ will always denote the underlying Euclidean or $l^2-$norm.

\textbf{Assumption (H):} There exist non-negative constants $C,\
\alpha,\ \beta,\ \theta$ such that for almost all $\omega$, the following
inequalities hold for all
$(x,y,z,t)\in\mathcal{O}\times\mathbb{R}\times\mathbb{R}^d\times\mathbb{R}^+$:\begin{enumerate}
\item   $|f(t,\omega,x,y,z)-f(t,\omega,x,y',z')|\leq C(|y-y'|+|z-z'|),$
\item $(\sum_{i=1}^d|g_i(t,\omega,x,y,z)-g_i(t,\omega,x,y',z')|^2)^{\frac{1}{2}}\leq
C|y-y'|+\alpha|z-z'|,$
\item $(|h(t,\omega,x,y,z)-h(t,\omega,x,y',z')|^2)^{\frac{1}{2}}\leq
C|y-y'|+\beta|z-z'|,$
\item $|l(t,\omega,x,y)-l(t,\omega,x,y')|\le \theta|y-y'|,$
\item the contraction property: $2\alpha+\beta^2+2\|Tr\|^2\theta<2\lambda.$
\end{enumerate}
\begin{remark} The last contraction property ensures existence and uniqueness for the solution of the SPDE without obstacle. (see \cite{DenisStoica})
\end{remark}

\textbf{Assumption (I):} Moreover we assume that for any
$T>0$,$$\xi\in L^2(\Omega\times\mathcal{O}) \text{ is an
$\mathcal{F}_0$-measurable random variable};$$
$$f(\cdot,\cdot,\cdot,0,0):=f^0\in
L^2([0,T]\times\Omega\times\mathcal{O};\bbR);$$
$$g(\cdot,\cdot,\cdot,0,0):=(g_1^0,...,g_d^0)\in
L^2([0,T]\times\Omega\times\mathcal{O};\bbR^d);$$
$$h(\cdot,\cdot,\cdot,0,0):=(h_1^0,...,h_i^0,...)\in
L^2([0,T]\times\Omega\times\mathcal{O};\bbR^{\mathbb{N}^*});$$
$$l(\cdot,\cdot,0):=l^0\in
L^2([0,T]\times\Omega\times\partial\mathcal{O};\bbR).$$

Now we introduce the notion of weak solution. For simplicity, we
fix the terminal time $T>0$. We denote by $\mathcal{H}_T$ the space
of $H^1(\mathcal{O})$-valued predictable processes
$(u_t)_{t\geq0}$ which satisfies
$$E\sup_{t\in[0,T]}\| u_t\|^2+E\int_0^T\mathcal{E}(u_t)dt<+\infty.$$
It is the natural space for solutions.

The space of test functions is denote by  $\mathcal{D}=\mathcal{C}%
_{c}^{\infty } (\R^+ )\otimes \mathcal{C}^2 (\bar{\cO} )$, where $\mathcal{C}%
_{c}^{\infty } (\R^+ )$ is the space of all real valued infinite
differentiable  functions with compact support in $\mathbb{R}^+$ and
$\mathcal{C}^2 (\bar{\cO} )$ the set
of functions in $C^2(\mathcal O)$ all of whose derivatives of order less than or equal to 2 have continuous extensions to $\bar {\cO}$.
\section{Linear SPDEs with Neumann boundary condition}\label{linearspdeneumann}
In this section, we recall some results concerning the linear SPDEs with Neumann boundary condition. First, consider the equation with regular coefficients and null Neumann boundary condition:
\begin{equation}\label{spde_zero_neumann_regular}
\left\{
\begin{aligned}
&du_t(x)=\big\{\cA u_t(x)+f(t,x)\big\}dt+h^i(t,x)dB^i_t,\quad\forall (t,x) \in \bbR^+\times \cO,\\
&\sum_{i,j}a_{i,j}(x)\partial_ju_t(x)n_i(x)=0,\quad\forall (t,x) \in \bbR^+\times \partial \cO,\\&u_0(x)=\xi,\quad\forall x \in \cO,
\end{aligned}
\right.
\end{equation}
with $f,h^i \in L^2([0,T]\times\Omega\times\mathcal{O};\bbR)$. According to Theorem 5.4 in \cite{DaPZ}, we know that (\ref{spde_zero_neumann_regular}) has a unique weak solution which can be written as:
$$u_t=P_t\xi+\int_0^tP_{t-s}f_sds+\int_0^tP_{t-s}h^i_sdB^i_s\,,\qquad a.s..$$
Note that in this case, $u$ belongs to $\cH_T$.

For more general case,
\begin{equation}\label{spde_linear_neumann}
\left \{
\begin{aligned}
&du_t(x)=\big\{\cA u_t(x)+{\partial}_i g^i(t,x)+f(t,x)\big\}dt+h^i(t,x)dB^i_t,\quad\forall (t,x) \in \bbR^+\times \cO,\\
&\sum_{i,j}(a_{i,j}(x)\partial_ju_t(x)+g^i_t(x))n_i(x)=l_t(x),\quad\forall (t,x) \in \bbR^+\times \partial \cO,\\&\ u_0(x)=\xi,\quad\forall x \in \cO,
\end{aligned}
\right.
\end{equation}
with $f,h^i,g^i \in L^2([0,T]\times\Omega\times\mathcal{O};\bbR)$ and $l \in L^2([0,T]\times\Omega\times\partial \mathcal{O};\bbR)$. By the linearity of the coefficients, it is easy to see that the weak solution $u$ can be decomposed into $u=u_1+u_2$, where $u_1$ is the solution of (\ref{spde_zero_neumann_regular}) and $u_2$ is the solution of PDEs without random coefficients:
\begin{equation}\label{spde_zero_neumann_deter}
\left \{
\begin{aligned}
&du_t(x)=\big \{\cA u_t(x)+{\partial}_i g^i(t,x)\big\}dt,\quad\forall (t,x) \in \bbR^+\times \cO,\\
&\sum_{i,j}(a_{i,j}(x)\partial_ju_t(x)+g^i_t(x))n_i(x)=l_t(x),\quad\forall (t,x) \in \bbR^+\times \partial \cO,\\&\ u_0(x)=\xi,\quad\forall x \in \cO.
\end{aligned}
\right.
\end{equation}
Now we assume that the coefficients $\xi$, $f$, $g$ and $h$ are smooth and $h=\sum_k h_k1_{[t_k,t_{k+1})}$ with $h_k$ being $\cF_{t_k}$-measurable and smooth in $x$. For the purpose of later discussion, we will show that the associated solution $u$ has a continuous version in $(t,x)$. First, let us recall some results of deterministic PDEs:
\begin{equation}\label{spde_zero_neumann_classical}
\left \{
\begin{aligned}
&du^2_t(x)=\big \{\cA u^2_t(x)+{\partial}_i g^i(t,x)\big\}dt,\quad\forall (t,x) \in \bbR^+\times \cO,\\
&\sum_{i,j}(a_{i,j}(x)\partial_ju^2_t(x)+g^i_t(x))n_i(x)=l_t(x),\quad\forall (t,x) \in \bbR^+\times \partial \cO,\\&u_0(x)=\xi,\quad\forall x \in \cO.
\end{aligned}
\right.
\end{equation}
According to Theorem 5.18 in \cite{Lieb}, if $\partial _i g^i\in C_{\alpha}$, $l \in C_{1+\alpha}$ and $\xi \in C_{2+\alpha}$, then the solution of (\ref{spde_zero_neumann_classical}) belongs to $C_{2+\alpha}$, where $C_{\alpha}$ is the H\"older space defined in Section 4.1 of \cite{Lieb}. This implies that the solution $u_2$ of (\ref{spde_zero_neumann_deter}) is continuous for almost all $\omega$. Due to the definition of $h^i$, it is easy to see that $P_{s}h^i_s(x)$  is continuous in space-time variables on any time interval $[t_k, t_{k+1})$. Hence, the It\^o integral $\int_0^t P_{t-s}h^i_s(x)dB^i_s$ is continuous in space-time variables and belongs to $C_{2+\alpha}$ for any $t$. Moreover, if $f\in C_\alpha$, the solution $u_1$ of (\ref{spde_zero_neumann_regular}) belongs to $C_{2+\alpha}$, $P$-a.s.. Thus, we can deduce that $u$ has a continuous version with respect to $(t,x)$.

\vspace{2mm}
At last,  It\^o's formula for the solution of \eqref{spde_linear_neumann} is proved.
\begin{lemma}\label{Ito}
Let $u$ be the solution of \eqref{spde_linear_neumann} and $\Phi:\mathbb{R}^+\times\mathbb{R}\rightarrow\mathbb{R}$ be a
function of class $\mathcal{C}^{1,2}$. We denote by $\Phi'$ and
$\Phi''$ the derivatives of $\Phi$ with respect to the space
variables and by $\frac{\partial\Phi}{\partial t}$ the partial
derivative with respect to time. We assume that these derivatives
are bounded. Let $u$ be a weak solution of (\ref{spde_linear_neumann}). Then, $P$-a.s. for all $t \in[0,T]$,
\begin{eqnarray*}
&&\hspace{-0.5cm}\int_\mathcal{O}\Phi(t,u_t(x))dx+\int_0^t\mathcal{E}(\Phi'(s,u_s),u_s)ds=\int_\mathcal{O}\Phi(0,\xi(x))dx+\int_0^t\int_\mathcal{O}\frac{\partial\Phi}{\partial
s}(s,u_s(x))dxds\\&&\hspace{-0.5cm}+\int_0^t(\Phi'(s,u_s),f_s)ds
-\sum_{i=1}^d\int_0^t\int_\mathcal{O}\Phi''(s,u_s(x))\partial_iu_s(x)g_s^i(x)dxds+\sum_{j=1}^{+\infty}\int_0^t(\Phi'(s,u_s),h_j)dB_s^j\\&&
\hspace{-0.5cm}+\int_0^t\int_{\partial \cO}\Phi(s,u_s(x))l_s(x)d\sigma(x)ds+\frac{1}{2}\sum_{j=1}^{+\infty}\int_0^t\int_\mathcal{O}\Phi''(s,u_s(x))(h_s^j(x))^2dxds.
\end{eqnarray*}
\end{lemma}
\begin{proof}
The proof is similar to Lemma 7 in \cite{DenisStoica} and Theorem 4.17 in \cite{DaPZ}. Thus we only give a sketch of it. Let $0=t_0<t_1<t_2<...<t_{n-1}<t_n=t$ be a partition of time interval $[0,t]$ and $\Delta:=\max_{i=1,2,...,n}|t_k-t_{k-1}|$. Then, $P-a.s.$,
\begin{equation*}
\begin{split}
&\int_\mathcal{O}\Phi(t,u_t(x))dx-\int_\mathcal{O}\Phi(0,\xi(x))dx
=\sum_{k=1}^{n}\Big( \int_\mathcal{O}\Phi(t_k,u_{t_k}(x))dx-\int_\mathcal{O}\Phi(t_{k-1},u_{t_{k-1}}(x))dx\Big)\\
&=\sum_{k=1}^{n}\Big\{\int_\mathcal{O}\frac{\partial \Phi}{\partial s}(t_{k-1},u_{t_{k-1}}(x))(t_k-t_{k-1})dx+\int_\mathcal{O}\Phi'(t_{k-1},u_{t_{k-1}}(x))(u_{t_k}(x)-u_{t_{k-1}}(x))dx\\
&+\int_\mathcal{O}\frac{1}{2}\Phi''(t_{k-1},u_{t_{k-1}}(x))(u_{t_k}(x)-u_{t_{k-1}}(x))^2dx\Big \}+o(\Delta).\\
\end{split}
\end{equation*}
Since $u$ is a weak solution of \eqref{spde_linear_neumann}, we have
\begin{equation*}
\begin{split}
&\int_\mathcal{O}\Phi'(t_{k-1},u_{t_{k-1}}(x))(u_{t_k}(x)-u_{t_{k-1}}(x))dx=-\int_{t_{k-1}}^{t_k}\mathcal{E}(\Phi'(t_{k-1},u_{t_{k-1}}),u_s)ds\\&+\int_{t_{k-1}}^{t_k}(\Phi'(t_{k-1},u_{t_{k-1}}),f_s)ds-\sum_{i=1}^d\int_{t_{k-1}}^{t_k}\int_\mathcal{O}\Phi''(t_{k-1},u_{t_{k-1}}(x))\partial_iu_{{t_{k-1}}}(x)g_s^i(x)dxds\\&+\sum_{j=1}^{+\infty}\int_0^t(\Phi'(t_{k-1},u_{t_{k-1}}),h_s^j)dB_s^j+\int_{t_{k-1}}^{t_k}\int_{\partial \cO}\Phi(t_{k-1},u_{t_{k-1}}(x))l_s(x)d\sigma(x)ds.
\end{split}
\end{equation*}
Thus, letting $\Delta \rightarrow +\infty$, we have, $P-a.s.$,
\begin{equation*}\sum_{k=1}^{n}\int_\mathcal{O}\frac{\partial \Phi}{\partial s}(t_{k-1},u_{t_{k-1}}(x))(t_k-t_{k-1})dx \rightarrow \int_0^t\int_\mathcal{O}\frac{\partial\Phi}{\partial s}(s,u_s(x))dxds,\end{equation*}
\begin{equation*}
\begin{split}
&\sum_{k=1}^{n}\Phi'(t_{k-1},u_{t_{k-1}}(x))(u_{t_k}(x)-u_{t_{k-1}}(x))dx\\
\rightarrow &-\int_0^t\mathcal{E}(\Phi'(s,u_s),u_s)ds+\int_0^t(\Phi'(s,u_s),f_s)ds
-\sum_{i=1}^d\int_0^t\int_\mathcal{O}\Phi''(s,u_s(x))\partial_iu_s(x)g_s^i(x)dxds\\&
+\sum_{j=1}^{+\infty}\int_0^t(\Phi'(s,u_s),h_s^j)dB_s^j+\int_0^t\int_{\partial \cO}\Phi(s,u_s(x))l_s(x)d\sigma(x)ds,
\end{split}
\end{equation*}
and
\begin{equation*}\sum_{k=1}^{n}\int_\mathcal{O}\frac{1}{2}\Phi''(t_{k-1},u_{t_{k-1}}(x))(u_{t_k}(x)-u_{t_{k-1}}(x))^2dx \rightarrow \frac{1}{2}\sum_{j=1}^{+\infty}\int_0^t\int_\mathcal{O}\Phi''(s,u_s(x))(h_s^j(x))^2dxds.\end{equation*}
Combine the three equalities above, we get the desired relation.
\end{proof}
\begin{remark}
As we have the It\^o formula, the following estimate for the solution of \eqref{spde_linear_neumann} can be easily deduced:
\begin{equation}\label{estimate:u}E\sup_t \|u_t\|^2+E\int_0^T\mathcal{E}(u_t)dt\leq CE\left(\|\xi\|^2+\int_0^T\|f_t\|^2+\||g_t|\|^2+\||h_t|\|^2+\|l_t\|^2dt\right),\end{equation}
where the constant $C$ only depends on the constants in Assumption (\begin{bf}H\end{bf}), $T$ and $d$. 
\end{remark}
\section{Parabolic potential analysis}{\label{capacity}}
\subsection{Parabolic capacity and potentials}
In this section we will recall some important definitions and
results concerning the obstacle problem for parabolic PDE in
\cite{Pierre} and \cite{PIERRE}.

$\mathcal{K}$ denotes $L^\infty(0,T;L^2(\mathcal{O}))\cap
L^2(0,T;H^1(\mathcal{O}))$ equipped with the norm:
\begin{eqnarray*}\|
v\|^2_\mathcal{K}&=&\|
v\|^2_{L^\infty(0,T;L^2(\mathcal{O}))}+\|
v\|^2_{L^2(0,T;H^1(\mathcal{O}))}\\
&=&\sup_{t\in[0,T[}\| v_t\|^2 +\int_0^T \left(
\| v_t \|^2  +\mathcal{E}(v_t)\right)\, dt
.\end{eqnarray*} $\mathcal{C}$ denotes the space of continuous
functions on compact support in $[0,T[\times\bar{\mathcal{O}}$ and
finally:
$$\mathcal {W}=\{\varphi\in L^2(0,T;H^1(\mathcal{O}));\ \frac{\partial\varphi}{\partial t}\in
L^2(0,T;H^{-1}(\mathcal{O}))\}, $$ endowed with the
norm$\|\varphi\|^2_{\mathcal {W}}=\|
\varphi\|^2_{L^2(0,T;H^1(\mathcal{O}))}+\|\displaystyle\frac{\partial
\varphi}{\partial t}\|^2_{L^2(0,T;H^{-1}(\mathcal{O}))}$.

It is known (see \cite{LionsMagenes}) that $\mathcal{W}$ is
continuously embedded in $C([0,T]; L^2 (\cO))$, the set of $L^2 (\cO
)$-valued continuous functions on $[0,T]$. Without ambiguity, we
will also consider
$\mathcal{W}_T=\{\varphi\in\mathcal{W};\varphi(T)=0\}$,
$\mathcal{W}^+=\{\varphi\in\mathcal{W};\varphi\geq0\}$,
$\mathcal{W}_T^+=\mathcal{W}_T\cap\mathcal{W}^+$.

We now introduce the notion of parabolic potentials and regular measures which permit to define the parabolic capacity.
\begin{definition}
An element $v\in \mathcal{K}$ is said to be a {\bf parabolic potential} if it satisfies:
$$ \forall\varphi\in\mathcal{W}_T^+,\
\int_0^T-(\frac{\partial\varphi_t}{\partial
t},v_t)dt+\int_0^T\mathcal{E}(\varphi_t,v_t)dt\geq0.$$
We denote by $\mathcal{P}$ the set of all parabolic potentials.
\end{definition}
The next representation property is  crucial:
\begin{proposition}(Proposition 1.1 in \cite{PIERRE})\label{presentation}
Let $v\in\mathcal{P}$, then there exists a unique positive Radon
measure on $[0,T[\times\bar{\mathcal{O}}$, denoted by $\nu^v$, such that:
$$\forall\varphi\in\mathcal{W}_T\cap\mathcal{C},\ \int_0^T(-\frac{\partial\varphi_t}{\partial t},v_t)dt+\int_0^T\mathcal{E}(\varphi_t,v_t)dt=\int_0^T\int_{\bar{\mathcal{O}}}\varphi(t,x)d\nu^v.$$
Moreover, $v$ admits a right-continuous (resp. left-continuous)
version $\hat{v} \ (\makebox{resp. } \bar{v}): [0,T]\mapsto L^2
(\cO)$ .\\
Such a Radon measure, $\nu^v$ is called {\bf a regular measure} and we write:
$$ \nu^v =\frac{\partial v}{\partial t}-\cA v .$$
\end{proposition}
\begin{remark} As a consequence, we can also define for all $v\in\mathcal{P}$:
$$ v_T =\lim_{t\uparrow T}\bar{v}_t \ \in L^2 (\cO ).$$
\end{remark}

\begin{definition}
Let $K\subset [0,T[\times\bar{\mathcal{O}}$ be compact, $v\in\mathcal{P}$
is said to be  \textit{$\nu-$superior} than 1 on $K$, if there
exists a sequence $v_n\in\mathcal{P}$ with $v_n\geq1\ a.e.$ on a
neighborhood of $K$ converging to $v$ in
$L^2(0,T;H^1(\mathcal{O}))$.
\end{definition}
We denote:$$\mathscr{S}_K=\{v\in\mathcal{P};\ v\ is\ \nu-superior\
to\ 1\ on\ K\}.$$
\begin{proposition}(Proposition 2.1 in \cite{PIERRE})
Let $K\subset [0,T[\times\bar{\mathcal{O}}$ compact, then $\mathscr{S}_K$
admits a smallest $v_K\in\mathcal{P}$ and the measure $\nu^v_K$
whose support is in $K$ satisfies
$$\int_0^T\int_{\bar{\mathcal{O}}}d\nu^v_K=\inf_{v\in\mathcal{P}}\{\int_0^T\int_{\bar{\mathcal{O}}}d\nu^v;\ v\in\mathscr{S}_K\}.$$
\end{proposition}
\begin{definition}(Parabolic Capacity)\begin{itemize}
                                        \item Let $K\subset [0,T[\times\bar{\mathcal{O}}$ be compact, we define
$cap(K)=\int_0^T\int_{\bar{\mathcal{O}}}d\nu^v_K$;
                                        \item let $O\subset
[0,T[\times\bar{\mathcal{O}}$ be open, we define $cap(O)=\sup\{cap(K);\
K\subset O\ compact\}$;
                                        \item   for any borelian
$E\subset [0,T[\times\bar{\mathcal{O}}$, we define $cap(E)=\inf\{cap(O);\
O\supset E\ open\}$.
                                      \end{itemize}

\end{definition}
\begin{definition}A property is said to hold quasi-everywhere (in short q.e.)
if it holds outside a set of null capacity.
\end{definition}
\begin{definition}(Quasi-continuous)

\noindent A function $u:[0,T[\times\bar{\mathcal{O}}\rightarrow\mathbb{R}$
 is called quasi-continuous, if there exists a decreasing sequence of open
subsets $O_n$ of $[0,T[\times\bar{\mathcal{O}}$ with: \begin{enumerate}
                        \item for all $n$, the restriction of $u$ to the complement of $O_n$ is
continuous;
                                       \item $\lim_{n\rightarrow+\infty}cap\;(O_n)=0$.
                                     \end{enumerate}
We say that $u$ admits a quasi-continuous version, if there exists
$\tilde{u}$ quasi-continuous  such that $\tilde{u}=u\ a.e.$.
\end{definition}
The next proposition, whose proof may be found in \cite{Pierre} or \cite{PIERRE} shall play an important role in the sequel:
\begin{proposition}\label{Versiont} Let $K\subset \bar{\cO}$ a compact set, then $\forall t\in [0,T[$
$$cap (\{ t\}\times K)=\lambda_d (K),$$
where $\lambda_d$ is the Lebesgue measure on $\cO$.\\
As a consequence, if $u: [0,T[\times \cO\rightarrow \R$ is a map defined quasi-everywhere then it defines uniquely a map from $[0,T[$ into $L^2 (\cO)$.
In other words, for any $t\in [0,T[$, $u_t$ is defined without any ambiguity as an element in $L^2 (\cO)$.
Moreover, if $u\in \mathcal{P}$, it admits  version $\bar{u}$ which is left continuous on $[0,T]$ with values in $L^2 (\cO )$ so that $u_T =\bar{u}_{T^-}$ is also defined without ambiguity.
\end{proposition}
\begin{remark} The previous proposition applies if for example $u$ is quasi-continuous.
\end{remark}

\begin{proposition}\label{approximation}(Theorem III.1 in \cite{PIERRE})
If $\varphi \in\mathcal{W}$, then it admits a unique quasi-continuous version that we denote by $\tilde{\varphi}$. Moreover, for all $v\in \mathcal{P}$,
the following relation holds:
$$\int_{[0,T[\times \bar{\cO}} \tilde{\varphi}d\nu^v =\int_0^T \left( -\partial_t \varphi ,v\right)+\mathcal{E}(\varphi ,v)\, dt +\left( \varphi_T ,v_{T}\right) .$$
\end{proposition}

\subsection{Applications to PDE's with obstacle}
For any function $\psi:[0,T[\times\mathcal{O}\rightarrow \bbR$ and
$u_0\in L^2(\mathcal{O})$, following  M. Pierre
\cite{Pierre,PIERRE}, F. Mignot and J.P. Puel \cite{MignotPuel}, we
define\begin{equation}\label{kappa}\kappa(\psi,u_0)= \text{ess}\inf
\{u\in\mathcal{P};\ u\geq\psi\ a.e.,\ u(0)\geq
u_0\}.\end{equation}This lower bound exists and is an element in
$\mathcal{P}$. Moreover, when $\psi$ is quasi-continuous, this
potential is the solution of the following reflected
problem:\begin{eqnarray*}\kappa\in\mathcal{P},\ \kappa\geq\psi,\
\frac{\partial\kappa}{\partial t}-\cA\kappa=0\text{ and }\frac{\partial \kappa}{\partial v_L}=0\ \mbox{on}\ \{u>\psi \},\
\kappa(0)=u_0.\end{eqnarray*} Mignot and Puel have proved in
\cite{MignotPuel} that $\kappa(\psi,u_0)$ is the limit (increasingly
and weakly in $L^2 (0,T; H^1 (\cO))$) when $\epsilon$ tends to $0$
of the solution of the following penalized equation
\begin{eqnarray*}u_\epsilon\in\mathcal{W},\ u_\epsilon(0)=u_0,\ \frac{\partial u_\epsilon}{\partial t}-
\cA u_\epsilon-\frac{(u_\epsilon-\psi)^-}{\epsilon}=0,\frac{\partial u_\epsilon}{\partial v_L}=0.\end{eqnarray*}
Let us point out that they obtain this result in the more general
case  where $\psi$ is only measurable from $[0,T[$ into
$L^2(\mathcal{O})$.\\For given $f\in L^2(0,T; H^{-1}(\mathcal{O}))$ and $l \in L^2(0,T;L^2(\partial O))$,
we denote by $\kappa_{u_0}^{f,l}$ the solution of the following
problem:$$\kappa\in\mathcal{W},\ \kappa(0)=u_0,\
\frac{\partial\kappa}{\partial t}-\cA\kappa=f,\frac{\partial \kappa}{\partial v_L}=l.$$
\\The
next theorem ensures  existence and uniqueness of the
solution of parabolic PDE with obstacle, it is proved in
\cite{Pierre}, Theorem 1.1. The proof  is based on a regularization argument of the obstacle, using the results of \cite{CT75}.

\begin{theorem}
Let $\psi:[0,T[\times\mathcal{O}\rightarrow\mathbb{R}$ be
quasi-continuous, suppose that there exists $\zeta\in\mathcal{P}$
with $|\psi|\leq\zeta\ a.e.$, $f\in L^2(0,T;H^{-1}(\mathcal{O})),l\in L^2(0,T;L^2(\partial \cO))$,
and the initial value $u_0\in L^2(\mathcal{O})$ with
$u_0\geq\psi(0)$, then there exists a unique
$u\in\kappa_{u_0}^{f,l}+\mathcal{P}$ quasi-continuous such that:
\begin{eqnarray*}u(0)=u_0,\ \tilde{u}\geq\psi,\ q.e.;\quad\int_0^T\int_{\bar{\mathcal{O}}}(\tilde{u}-\tilde{\psi})d\nu^{u-\kappa_{u_0}^{f,l}}=0\end{eqnarray*}
\end{theorem}

%
We end this section by a convergence lemma which plays an important role in our approach (Lemma 3.8 in \cite{PIERRE}):
\begin{lemma}\label{convergemeas}
If $v^n\in\mathcal{P}$ is a bounded sequence in $\mathcal{K}$ and
converges weakly to $v$ in $L^2(0,T;H^1(\mathcal{O}))$; if $u$ is
a quasi-continuous function and $|u|$ is bounded by a element in
$\mathcal{P}$. Then
$$\lim_{n\rightarrow+\infty}\int_0^T\int_{\bar{\mathcal{O}}}ud\nu^{v^n}=\int_0^T\int_
{\bar{\mathcal{O}}}ud\nu^{v}.$$
\end{lemma}
\begin{remark}For the more general case one can see \cite{PIERRE} Lemma 3.8. \end{remark}

\section{Quasi--continuity of the solution of SPDE without obstacle}\label{quasi-contSPDE}

Under assumptions {\bf(H)} and {\bf (I)}, SPDE (\ref{spde_linear_neumann}) with  Neumann boundary condition admits a
unique  solution  in $\mathcal{H}_T$. We denote it by
$\mathcal{U}(\xi,f,g,h,l)$.

The main  theorem of this section is the following:
\begin{theorem}\label{mainquasicontinuity}
 Under assumptions {\bf(H)} and {\bf (I)}, $u=\mathcal{U}(\xi,f,g,h,l)$ the
solution of SPDE (\ref{spde_linear_neumann})  admits a quasi-continuous version
denoted by $\tilde{u}$ i.e.
 $u=\tilde{u}$ $dP\otimes dt\otimes dx -$a.e. and for almost all $w\in\Omega$,
 $(t,x)\rightarrow \tilde{u}_t
(w,x)$ is quasi-continuous.
\end{theorem}
Before giving the proof of this theorem, we need the following lemmas. The first one is proved in \cite{PIERRE}, Lemma 3.3:
\begin{lemma}\label{cap}There exists $C>0$ such that, for all open set $\vartheta\subset [0,T[\times \bar{\mathcal{O}}$ and  $v\in\mathcal{P}$ with $v\geq1\ a.e.$ on $\vartheta$:$$cap\vartheta\leq C\| v\|^2_{\mathcal{K}}.$$\end{lemma}
Let $\kappa=\kappa(u,u^+(0))$ be defined by relation \eqref{kappa}. One has to note that $\kappa$ is a random function. From now on, we always take for $\kappa$ the following measurable version
$$\kappa =\sup_n v^n,$$
where  $(v^n)$ is the non-decreasing sequence of random functions given by
\begin{equation} \label{eq:1}
\left\{ \begin{split}
         &\frac{\partial v_t^n}{\partial t}=\cA v_t^n+n(v_t^n-u_t)^-,\\
         &\frac{\partial v^n}{\partial v_L}=0,\ v^n_0=u^+(0).
                          \end{split} \right.
                          \end{equation}
Using the results recalled in Subsection \ref{capacity}, we know that for almost all $w\in\Omega$, $v^n (w)$ converges weakly to $v(w)=\kappa(u(w),u^+(0)(w))$ in
$L^2(0,T;H^1(\mathcal{O}))$ and that $v\geq u$.
\begin{lemma}\label{estimoftau} We have the following estimate:
\begin{eqnarray*}E\|\kappa\|_{\mathcal{K}}^2 \, \leq C\left(E\| u_0^+\|^2+E\| u_0\|^2+E\int_0^T\| f^0_t\|^2+\||g^0_t|\|^2+\||h^0_t|\|^2+\|l^0_t\|_{L^2(\partial \cO)}^2dt\right),\end{eqnarray*}
where $C$ is a constant depending only on the structure constants of the equation.
\end{lemma}
\begin{proof} All along this proof, we shall denote by $C$ or $C_{\epsilon} $ some constant which may change from line to line.


Consider the approximation $(v^n )_n$ defined by \eqref{eq:1},
$P$-almost surely, it converges weakly to $v=\kappa(u,u^+(0))$ in
$L^2(0,T;H^1(\mathcal{O}))$. We remark that $v^n-u$ satisfies
the following
equation:
\begin{equation*}
\left \{
\begin{aligned}d(v^n_t-u_t)-\cA(v_t^n-u_t)dt=&-f_tdt-\sum_{i=1}^d\partial_ig^i_tdt-\sum_{j=1}^{+\infty}h^j_tdB^j_t+n(v_t^n-u_t)^-dt,\\
\sum_{i,j}(a_{i,j}\partial_j(v^n_t-u_t)-g^i_t)n_i=&-l_t.\end{aligned}
\right.
\end{equation*}
Applying It\^o's
formula (Lemma \ref{Ito}) to $(v^n-u)^2$, we have almost surely, for all $t\in[0,T]$,
\begin{equation}\label{original}
\begin{split}&\| v_t^n-u_t\|^2+2\int_0^t\mathcal{E}(v_s^n-u_s)ds=\| u^+_0-u_0\|^2-2\int_0^t(v_s^n-u_s,f_s)ds\\&+2\sum_{i=1}^d\int_0^t(\partial_i(v_s^n-u_s),g^i_s)ds+\int_0^t\| |h_s|\|^2ds-2\int_0^t(l_s,v_s^n-u_s)_{L^2(\partial \cO)}ds\\&-2\sum_{j=1}^{+\infty}\int_0^t(v_s^n-u_s,h^j_s)dB^j_s+2\int_0^t(n(v_s^n-u_s)^-,v_s^n-u_s)ds.\end{split}\end{equation}
The last term in the right member of (\ref{original}) is obviously non-positive
and noting that
$$(l_s,v_s^n-u_s)_{L^2(\partial \cO)}\le \|l_s\|_{L^2(\partial \cO)}\|v_s^n-u_s\|_{L^2(\partial \cO)}\le\|Tr\|\|l_s\|_{L^2(\partial \cO)}\|v_s^n-u_s\|_{H^1(\cO)},$$
then, taking expectation and using Cauchy-Schwartz's inequality, we get
\begin{equation*}\begin{split}
&E\| v_t^n-u_t\|^2+(2-\frac{\epsilon}{\lambda})E\int_0^t\mathcal{E}(v_s^n-u_s)ds\\
\leq&\, E\| u_0^+-u_0\|^2+CE\int_0^t\| v_s^n-u_s\|^2ds+C_{\epsilon}E\int_0^t\|l(s,u_s)\|_{L^2(\partial \cO)}^2ds
\\+&\,E\int_0^t\| f_s(u_s,\nabla u_s)\|^2ds+C_\epsilon E\int_0^t\| |g_s(u_s,\nabla u_s)|\|^2ds+E\int_0^t\| |h_s(u_s,\nabla u_s)|\|^2ds.
\end{split}\end{equation*}
Therefore, by using the Lipschitz conditions on the coefficients, we have
\begin{equation*}\begin{split}&E\| v_t^n-u_t\|^2+(2-\frac{\epsilon}{\lambda})E\int_0^t\mathcal{E}(v_s^n-u_s)ds\\
\leq&\, E\| u_0^+-u_0\|^2+CE\int_0^t\| v_s^n-u_s\|^2ds+C\Big(E\int_0^t\| u_s\|^2ds+E\int_0^t\mathcal{E}(u_s)ds\Big)
\\+&\,CE\int_0^t(\| f_s^0\|^2+\| |g_s^0|\|^2+\| |h_s^0|\|^2+\|l_s^0\|_{L^2(\partial \cO)}^2)ds\,.\end{split}\end{equation*}
Combining with (\ref{estimate:u}), this yields
\begin{equation*}
\begin{split}
E\| v_t^n-u_t\|^2&+(2-\frac{\epsilon}{\lambda})E\int_0^t\mathcal{E}(v_s^n-u_s)ds\leq E\| u_0^+-u_0\|^2+CE\int_0^t\| v_s^n-u_s\|^2ds\\&+\ CE\Big(\| u_0\|^2+\int_0^T(\|f_t^0\|^2+\| |g_t^0|\|^2+\||h_t^0|\|^2+\|l_s^0\|_{L^2(\partial \cO)}^2)dt\Big).
\end{split}
\end{equation*}
We  take now  $\epsilon$ small enough such that $(2-\frac{\epsilon}{\lambda})>0$, then, with
Gronwall's lemma, we obtain for each $t\in [0,T]$
\begin{equation*}\begin{split}E\| v_t^n-u_t\|^2\leq\,Ce^{c'T}\Big(E\| u_0^+&-u_0\|^2+E\| u_0\|^2
\\&+\,E\int_0^T\| f_t^0\|^2+\| |g_t^0|\|^2+\| |h_t^0|\|^2+\|l_t^0\|_{L^2(\partial \cO)}^2dt\Big).\end{split}\end{equation*}
As we a priori know that $P$-almost surely, $(v^n)_n$ tends to $\kappa$ strongly in $L^2 ([0,T]\times \cO)$, the previous estimate yields, thanks to the dominated convergence theorem,  that $(v^n )_n$ converges to
$\kappa$ strongly in $L^2 (\Omega \times [0,T]\times \cO)$ and
\begin{equation*}
\begin{split}
\sup_{t\in [0,T]}E\| \kappa_t-u_t\|^2\leq\,Ce^{c'T}\Big(E\| u_0^+&-u_0\|^2+E\| u_0\|^2\\&+\,E\int_0^T\| f_t^0\|^2+\| |g_t^0|\|^2+\| |h_t^0|\|^2+\|l_t^0\|_{L^2(\partial \cO)}^2dt\Big).
\end{split}\end{equation*}
Moreover, as $(v^n)_n$ tends to $\kappa$ weakly in $L^2 ([0,T]; H^1 (\cO ))$ $P$-a.s., we have for all $t\in [0,T]$:
\begin{equation*}
\begin{split}&E \int_0^T \mathcal{E}(\kappa_s -u_s)ds\leq\liminf_n E\int_0^T \mathcal{E}(v_s^n-u_s)ds\\ &\leq\,TCe^{c'T}\Big(E\| u_0^+-u_0\|^2+E\| u_0\|^2+E\int_0^T\| f_t^0\|^2+\| |g_t^0|\|^2+\| |h_t^0|\|^2+\|l_t^0\|_{L^2(\partial \cO)}^2dt\Big).\end{split}\end{equation*}
Let us now study the stochastic term. Define the martingales
$$M^n_t =\sum_{j=1}^{+\infty}\int_0^t(v_s^n-u_s,h^j_s)dB^j_s\ \makebox{ and }\ M_t=\sum_{j=1}^{+\infty}\int_0^t(\kappa_s-u_s,h^j_s)dB^j_s .$$
Then
\begin{eqnarray*}
E[|M^n_T -M_T|^2]=E\int_0^T \sum_{j=1}^{+\infty}(\kappa_s -v^n_s
,h_s )^2 \, ds\leq
E\int_0^T\|\kappa_s-v^n_s\|^2\|
|h_s|\|^2ds.
\end{eqnarray*}
Using the strong convergence of $(v^n )$ to $\kappa$ we conclude that $(M^n)$ tends to $M$ in $L^2$ sense.
Passing to the limit, we get, almost surely, for all $t\in[0,T]$,
\begin{equation*}
\begin{split}\|\kappa_t-u_t\|^2+2\int_0^t\mathcal{E}(\kappa_s -u_s)ds&\leq\| u^+_0-u_0\|^2-2\int_0^t(\kappa_s-u_s,f_s)ds\\&+2\sum_{i=1}^d\int_0^t(\partial_i(\kappa_s-u_s),g^i_s)ds-2\int_0^t(\kappa_s-u_s,l_s)_{L^2(\partial \cO)}ds\\&-
2\sum_{j=1}^{+\infty}\int_0^t(\kappa_s-u_s,h^j_s)dB^j_s+\int_0^t\| |h_s|\|^2ds.\end{split}\end{equation*}
As a consequence of Burkholder-Davies-Gundy's inequalities, we get
\begin{equation*}
\begin{split}
E\sup_{t\in[0,T]}|\sum_{j=1}^{+\infty}\int_0^t(\kappa_s-u_s,h^j_s) dB^j_s|&\leq
CE[\int_0^T\sum_{j=1}^{+\infty}(\kappa_s-u_s,h^j_s)^2ds]^{1/2}\\&\leq CE[\sup_{t\in[0,T]}\|
\kappa_t-u_t\|(\int_0^T\| |h_t |\|^2dt)^{1/2}]\\&\leq\epsilon E\sup_{t\in[0,T]}\|
\kappa_t -u_t\|^2+C_\epsilon E\int_0^T\| |h_t |\|^2dt.
\end{split}
\end{equation*}
By Lipschitz conditions
on $h$ and (\ref{estimate:u}) this yields
\begin{equation*}
\begin{split}
E\sup_{t\in[0,T]}|\sum_{j=1}^{+\infty}\int_0^t(\kappa_s-u_s,h_s)dB_s|&\leq\epsilon
E\sup_{t\in[0,T]}\| \kappa_t-u_t\|\|^2+C\Big(E\|
u_0\|^2\\&+E\int_0^T\big(\| f_t^0\|^2+\| |g_t^0|\|^2+\| |h_t^0|\|^2+\|l_t^0\|_{L^2(\partial \cO)}^2\big)dt\Big).
\end{split}
\end{equation*}
Hence,
\begin{equation*}
\begin{split}
(1-\epsilon)E\sup_{t\in[0,T]}\| \kappa_t&-u_t\|^2+(2-\frac{\epsilon}{\lambda})E\int_0^T\mathcal{E}(\kappa_t-u_t)dt\\
&\leq\,C\Big(E\| u_0^+-u_0\|^2+E\int_0^T\big(\| f_t^0\|^2+\| |g_t^0|\|^2+\| |h_t^0|\|^2+\|l_t^0\|_{L^2(\partial \cO)}^2\big)dt\Big).
\end{split}
\end{equation*}
 We can take $\epsilon$ small
enough such that $1-\epsilon>0$ and $2-\frac{\epsilon}{\lambda}>0$,
hence, \begin{equation*}\begin{split}
E\sup_{t\in[0,T]}\| \kappa_t
&-u_t\|^2+E\int_0^T\mathcal{E}(\kappa_t-u_t)dt\\
\leq&C\Big(E\| u_0^+-u_0\|^2+E\|
u_0\|^2+E\int_0^T\big(\| f_t^0\|^2+\| |g_t^0|\|^2+\| |h_t^0|\|^2+\|l_t^0\|_{L^2(\partial \cO)}^2\big)dt\Big).
\end{split}\end{equation*}
Then, combining with
(\ref{estimate:u}), we get the desired estimate:
\begin{equation*}
\begin{split}
E\sup_{t\in[0,T]}\| \kappa_t\|^2&+E\int_0^T\mathcal{E}(\kappa_t)dt\\ \leq&\,C\Big(E\| u_0^+\|^2
+E\int_0^T\big(\| f_t^0\|^2+\| |g_t^0|\|^2+\| |h_t^0|\|^2+\|l_t^0\|_{L^2(\partial \cO)}^2\big)dt\Big).
\end{split}
\end{equation*}
\end{proof}

\textbf{Proof of Theorem \ref{mainquasicontinuity}:}
Let $(u_0^n)_n$, $(f^n )_n$, $(g^n )_n$, $(h^n )_n$  and $(l^n)_n$ are
sequences of  smooth elements in $L^2(\Omega\times[0,T];L^2(\cO))$ and $L^2(\Omega\times[0,T];L^2(\partial \cO))$ respectively
which converge respectively to $u_0$, $f$, $g$, $h$ and $l$ and $h=\sum_k h_k1_{[t_k,t_{k+1})}$ with $h_k$ being $\cF_{t_k}$-measurable. For all $n\in\mathbb{N}^*$
we define $u^n=\cU(u_0^n,f^n,g^n,h^n,l^n)$. From Section \ref{linearspdeneumann}, we know that $u^n$ is $P$-almost surely continuous in $(t,x)$. Then we can do a similar argument as in \cite{DMZ12} to end the proof of Theorem \ref{mainquasicontinuity}.   $\hspace{12.8cm}\Box$

\section{Existence and uniqueness result}
\subsection{Weak solution}
\textbf{Assumption (O):} The obstacle $S$ is assumed to be an adapted process, quasi-continuous, such that $S_0 \leq \xi$ $P$-almost surely and
controlled by the solution of a SPDE, i.e. $\forall t\in[0,T]$,
\begin{equation}S_t\leq S'_t\end{equation} where
$S'$ is the solution of the linear SPDE
\begin{equation}\label{obstacle}
\left\{ \begin{split}
 &dS'_t=\cA S'_tdt+f'_tdt+\sum_{i=1}^d \partial_i g'_{i,t}dt+\sum_{j=1}^{+\infty}h'_{j,t}dB^j_t,\\
 &\sum_{i,j} (a_{i,j}\partial_jS'_t(x)+g'_i)n_i=l'(t,x),\ \  S'(0)=S'_0,
\end{split}\right. \end{equation}
where $S'_0\in L^2 (\Omega\times \cO)$ is
$\mathcal{F}_0$-measurable, $f'$, $g'$, $h'$ and $l'$ are adapted
processes respectively in $L^2
([0,T]\times\Omega\times\cO;\mathbb{R})$,  $L^2
([0,T]\times\Omega\times\cO;\mathbb{R}^d)$, $L^2
([0,T]\times\Omega\times\cO;\mathbb{R}^{\mathbb{N}^*})$ and $L^2
([0,T]\times\Omega\times \partial \cO;\mathbb{R})$.
\begin{remark}Here again, we know that $S'$ uniquely exists and satisfies the following estimate:
\begin{equation}\label{estimobstacle}E\sup_{t\in[0,T]}\| S'_t\|^2+E\int_0^T\mathcal{E}(S'_t)dt\leq CE\left[\| S'_0\|^2+\int_0^T(\| f'_t\|^2+\| |g'_t|\|^2+\| |h'_t|\|^2+\|l'_t\|_{L^2(\partial \cO)}^2)dt\right].\end{equation}
Moreover, from Theorem \ref{mainquasicontinuity}, $S'$ admits a
quasi-continuous version. \\
Let us also remark that even if this assumption seems restrictive, since $S'$ is driven by the same operator and Brownian motions as $u$, it encompasses
a large class of examples. \end{remark}
We now are able to define rigorously the notion of solution to the problem with obstacle we consider.
\begin{definition} A pair
$(u,\nu)$ is said to be a solution of OSPDE \eqref{SPDEO} if
\begin{enumerate}
    \item $u\in\mathcal{H}_T$ and $u(t,x)\geq S(t,x),\ dP\otimes dt\otimes
    dx-a.e.$ and $u_0(x)=\xi,\ dP\otimes dx-a.e.$;
    \item $\nu$ is a random regular measure defined on
    $[0,T)\times\bar{\mathcal{O}}$;
    \item the following relation holds almost surely, for all
    $t\in[0,T]$ and $\forall\varphi\in\mathcal{D}$,
     \begin{equation}\label{solution}
     \begin{split}&\hspace{-0.3cm}(u_t,\varphi_t)-(\xi,\varphi_0)-\int_0^t(u_s,\partial_s\varphi_s)ds+\int_0^t\mathcal{E}(u_s,\varphi_s)ds+\sum_{i=1}^d\int_0^t(g^i_s,\partial_i\varphi_s)ds\nonumber\\
    &\hspace{-0.6cm}=\int_0^t(f_s,\varphi_s)ds+\int_0^t\int_{\partial\mathcal{O}}l_s(x)\varphi_s(x)d\sigma(x)ds+\sum_{j=1}^{+\infty}\int_0^t(h^j_s,\varphi_s)dB^j_s+\int_0^t\int_{\bar{\mathcal{O}}}\varphi_s(x)\nu(dx,ds);
    \end{split}\end{equation}
    \item $u$ admits a quasi-continuous version, $\tilde{u}$, and we have  $$\int_0^T\int_{\bar{\mathcal{O}}}(\tilde{u}(s,x)-{S}(s,x))\nu(dx,ds)=0,\ \
    a.s..\\$$
  \end{enumerate}
\end{definition}

In order to apply the penalized method to obstacle problem \eqref{SPDEO}, we need the following comparison theorem for the solution of \eqref{spde_linear_neumann}, which can be easily obtained by It\^o's formula.
\begin{theorem}
Let $\xi',f',l'$ be similar to $\xi,f,l$ and let $u$ be the solution of \eqref{spde_linear_neumann} corresponding to $(\xi,f,g,h,l)$ and $u'$ be the solution corresponding to $(\xi',f',g,h,l')$. Assume that the following conditions hold:
\begin{enumerate}
\item $\xi\leq\xi',\ dx\otimes dP-a.e.$;
\item $f(u,\nabla u)\leq f'(u,\nabla u),\ dt\otimes dx\otimes dP-a.e.$;
\item $l(u)\leq l'(u),\ dx\otimes dP-a.e.$.
\end{enumerate}
Then, for all most all $\omega$, $u(t,x)\leq u'(t,x),\ q.e.$.
\end{theorem}

As we have It\^o's formula and comparison theorem for the solution of SPDE \eqref{SPDEO} without obstacle, we can make a similar argument as Section 5 in \cite{DMZ12}. More precisely, we first establish the result in the linear case. Then, we prove an It\^o formula for the difference of two solutions of (linear) OSPDEs and finally conclude the following result thanks to a Picard iteration procedure:
\begin{theorem}{\label{maintheo}}
Under assumptions {\bf (H)}, {\bf (I)} and {\bf (O)}, there exists a unique weak solution of OSPDE \eqref{SPDEO} associated to $(\xi,f,g,h,l,S)$.
We denote by $\mathcal{R}(\xi,f,g,h,l,S)$ the solution of OSPDE \eqref{SPDEO} when it exists and is unique.
\end{theorem}

\subsection{It\^o's formula}
We can establish the following It\^o formula for the solution of OSPDE \eqref{SPDEO}. The proof is similar to that in \cite{DMZ12}.
\begin{theorem}\label{Itoformula}
Let $u$ be the solution of OSPDE \eqref{SPDEO} and
$\Phi:\mathbb{R}^+\times\mathbb{R}\rightarrow\mathbb{R}$ be a
function of class $\mathcal{C}^{1,2}$. We denote by $\Phi'$ and
$\Phi''$ the derivatives of $\Phi$ with respect to the space
variables and by $\frac{\partial\Phi}{\partial t}$ the partial
derivative with respect to time. We assume that these derivatives
are bounded. Then $P$-a.s. for
all $t\in[0,T]$,
\begin{equation*}\begin{split}
\int_\mathcal{O}\Phi(t,u_t(x))dx&+\int_0^t\mathcal{E}(\Phi'(s,u_s),u_s)ds=\int_\mathcal{O}\Phi(0,\xi(x))dx+\int_0^t\int_\mathcal{O}\frac{\partial\Phi}{\partial
s}(s,u_s(x))dxds\\&+\int_0^t(\Phi'(s,u_s),f_s)ds
-\sum_{i=1}^d\int_0^t\int_\mathcal{O}\Phi''(s,u_s(x))\partial_iu_s(x)g_i(x)dxds\\&+\sum_{j=1}^{+\infty}\int_0^t(\Phi'(s,u_s),h_j)dB_s^j
+\int_0^t\int_{\partial \cO}\Phi(s,u_s(x))l_s(x)d\sigma(x)ds\\&
+\frac{1}{2}\sum_{j=1}^{+\infty}\int_0^t\int_\mathcal{O}\Phi''(s,u_s(x))(h_{j,s}(x))^2dxds+\int_0^t\int_\mathcal{O}\Phi'(s,\tilde{u}_s(x))\nu(dxds).
\end{split}\end{equation*}
\end{theorem}
\subsection{Comparison theorem}
The above It\^o formula naturally leads to a comparison theorem, the proof is similar to the Dirichlet case (see Theorem 8 in \cite{DMZ12}). More precisely, consider $(u^1,\nu^1)=\mathcal{R} (\xi^1,f^1,g,h,l^1,S^1)$ the solution of the SPDE with obstacle
\begin{equation*}\left\{ \begin{split}&du^1_t(x)=\cA u^1_t(x)dt+f^1(t,x,u^1_t(x),\nabla
u^1_t(x))dt+\sum_{i=1}^d\partial_ig_i(t,x,u^1_t(x),\nabla
u^1_t(x))dt\\&\ \ \ \ \quad \quad+\sum_{j=1}^{+\infty}h_j(t,x,u^1_t(x),\nabla
u^1_t(x))dB^j_t +\nu^1(x, dt),\\
&\sum_{i,j}(a_{i,j}\partial_j u^1_t(x)+g^i_t(x))n_i=l^1(t,x),\\
&u^1\geq S^1\,,\  u^1_0=\xi^1,
\end{split}\right.\end{equation*}
where we assume $(\xi^1,f^1,g,h,l^1,S^1)$ satisfy hypotheses {\bf (H)}, {\bf (I)} and {\bf (O)}.

We consider another coefficient $f^2$ which satisfies the same assumptions as $f^1$, another obstacle $S^2$ which satisfies {\bf (O)}, another initial condition $\xi^2$ belonging to $L^2 (\Omega \times \cO)$ and $\mathcal{F}_0$ adapted such that $\xi^2\geq S^2_0$ and another boundary coefficient $l^2$ satisfying the same assumption as $l^1$. We denote by $(u^2,\nu^2 )=\mathcal{R}(\xi^2, f^2,g,h,l^2,S^2)$.
%
\begin{theorem} \label{comparison}Assume that the following conditions
hold\begin{enumerate}
      \item $\xi^1\leq\xi^2,\ dx\otimes dP-a.e.$;
      \item $f^1(u^1,\nabla u^1)\leq f^2(u^1,\nabla u^1),\ dt\otimes dx\otimes dP-a.e.$;
      \item $l^1(u)\leq l^2(u),\ dt\otimes dx\otimes dP-a.e.$;
      \item $S^1\leq S^2,\ dt\otimes dx\otimes dP-a.e.$.
    \end{enumerate}
Then, for almost all $\omega\in\Omega$, $u^1(t,x)\leq u^2(t,x),\
q.e..$\end{theorem}
\begin{remark}Applying the comparison theorem to the same obstacle  gives another proof of  the uniqueness of the solution.\end{remark}

%


\begin{thebibliography}{99}
\addcontentsline{toc}{chapter}{Bibliographie}
\bibitem{Ada} Adams R.A.: {Sobolev spaces}. \textit{Academic Press}, 1975

 \bibitem{Aronson}  Aronson D.G.: On the Green's function for second order parabolic
differential equations with discontinuous coefficients. \textit{ Bulletin of
the American Mathematical Society}, {\bf 69},  841-847 (1963).

\bibitem{Aronson3}  Aronson, D.G.: Non-negative solutions
of linear parabolic equations, Annali della Scuola Normale
Superiore di Pisa, Classe di Scienze 3, {\bf tome 22} (4), 607-694  (1968).


\bibitem{BCEF}  Bally V.,  Caballero E., El-Karoui N. and
Fernandez, B.: Reflected BSDE's PDE's and Variational Inequalities. \textit{preprint INRIA report}, (2004).
%

\bibitem{BensoussanLions78} Bensoussan  A. and Lions J.-L.: Applications des Inéquations variationnelles en contrôle
stochastique. \textit{Dunod}, Paris (1978).


\bibitem{CT75}  Charrier P. and Troianiello G.M.
Un r\'esultat d'existence et de r\'egularit\'e pour les
solutions fortes d'un probl\`eme unilat\'eral d'\'evolution avec
obstacle d\'ependant du temps. \textit{ C.R.Acad. Sc. Paris}, {\bf 281}, s\'erie A,
621 (1975).

\bibitem{DaPZ}  Da Prato G. and Zabczyk J. :{Stochastic equations in infinite dimensions}. \textit{Cambridge University Press}, (1992).

\bibitem{DMZ}  Dalang R.C., Mueller C.,  and Zambotti L.:
Hitting properties of parabolic s.p.d.e.'s with reflection.
\textit{The Annals of Probability}, {\bf34}, No.4, 1423-1450 (2006).

\bibitem{Denis} Denis L.: Solutions of SPDE considered as Dirichlet
Processes. \textit{Bernoulli Journal of Probability}, {\bf 10} (5), 783-827 (2004).

\bibitem{DenisStoica}  Denis L. and Sto\"{\i}ca L.:  A general analytical
result for non-linear s.p.d.e.'s and applications. 
\textit{Electronic Journal of Probability}, {\bf 9}, 674-709 (2004).

%
%

\bibitem{DMZ12}  Denis L., Matoussi A. and Zhang J.: The obstacle problem for quasilinear stochastic PDEs: Analytical approach.
\textit{The Annals of Probability}, {\bf 42}, 865--905 (2014).

\bibitem{DonatiPardoux} Donati-Martin C. and Pardoux
E.: White noise driven SPDEs with reflection. \textit{Probability
Theory and Related Fields}, {\bf 95}, 1-24 (1993).

\bibitem{EKPPQ}  El Karoui N., Kapoudjian C., Pardoux E., Peng S., and Quenez M.C.: Reflected Solutions of
Backward SDE and Related Obstacle Problems for PDEs. \textit{The Annals of
Probability}, {\bf 25} (2), 702-737 (1997).

\bibitem{GT} Gilbarg, D. and Trudinger, N. S.:{Elliptic partial differential equations of second order}. \textit{Classics in mathematics, Springer}, (2001).


\bibitem{Klimsiak} Klimsiak T.: Reflected BSDEs and
obstacle problem for semilinear PDEs in divergence form. \textit{Stochastic
Processes and their Applications},  {\bf 122} (1), 134-169
(2012).

\bibitem{Lieb} Lieberman G. M.: Second Order Parabolic Differential Equations.\textit{World Scientific}, (1996).

\bibitem{LionsMagenes} Lions J.L. and Magenes
E.: Probl\`emes aux limites non homog\`enes et applications.
\textit{Dunod, Paris}, (1968).


\bibitem{MatoussiStoica} Matoussi A. and Sto\"{\i}ca
L.: The Obstacle Problem for Quasilinear Stochastic PDE's.
\textit{The Annals of Probability}, {\bf 38}, No.3, 1143-1179 (2010).

\bibitem{MignotPuel} Mignot F. and Puel J.P.:  In\'equations d'\'evolution paraboliques avec convexes d\'ependant
du temps. Applications aux in\'equations quasi-variationnelles
d'\'evolution. \textit{Arch. for Rat. Mech. and Ana.},  {\bf{64}}, No.1,  59-91 (1977).

\bibitem{NualartPardoux} Nualart D. and Pardoux
E.: White noise driven quasilinear SPDEs with reflection.
\textit{Probability Theory and Related Fields},   {\bf 93}, 77-89 (1992).

\bibitem{Pierre} Pierre M.: Probl\`emes d'Evolution avec Contraintes Unilaterales et
Potentiels Parabolique. \textit{Comm. in Partial Differential Equations},
{\bf 4} (10), 1149-1197 (1979).

\bibitem{PIERRE} Pierre M.: Repr\'esentant Pr\'ecis d'Un Potentiel Parabolique. \textit{S\'eminaire
de Th\'eorie du Potentiel, Paris}, {\bf No.5}, Lecture Notes in Math. 814, 186-228 (1980).


\bibitem{RiezNagy}  Riesz, F.  and  Nagy, B.: Functional Analysis.  \textit{Dover, New York}, 1990.


\bibitem{SW} Sanz M. and Vuillermot P.: Equivalence and Hölder Sobolev regularity of solutions for a class of non-autonomous stochastic partial differential equations. \textit{Ann. I. H. Poincar\'e}, {\bf{39}} (4), 703-742, (2003).


\bibitem{XuZhang} Xu T.G. and Zhang
T.S.: White noise driven SPDEs with reflection: Existence,
uniqueness and large deviation principles.
\textit{Stochatic processes and their applications}, {\bf 119}, 3453-3470 (2009).

\bibitem{Zhang} Zhang T.S.:
White noise driven SPDEs with reflection: Strong Feller properties and Harnack inequalities.
\textit{Potential Analysis}, {\bf 33} (2),  137-151 (2010).



\end{thebibliography}
\end{document}